\documentclass[preprint,3p,times]{elsarticle}

\usepackage{amssymb}
\usepackage{amsmath}
\usepackage{amsthm}
\usepackage{mathrsfs}
\usepackage{yhmath}
\usepackage{color}
\usepackage{enumitem}
\usepackage{verbatim}
\usepackage{tikz}
\usepackage{appendix}
\usepackage[cal=cm]{mathalfa}
\usepackage{xspace} 
\usepackage{cases}
\usepackage{graphicx}
\usepackage{tikz}
\usepackage{float}
\usepackage[colorlinks=true]{hyperref}

\newtheorem{theorem}{Theorem}[section]
\newtheorem{lemma}[theorem]{Lemma}
\newtheorem{prop}[theorem]{Proposition}

\newtheorem{mainthm}{Theorem}

\newtheorem{maincor}[mainthm]{Corollary}

\theoremstyle{definition}

\theoremstyle{remark}

\newcommand{\abs}[1]{\lvert#1\rvert} %长度
\newcommand{\INF}{[\omega]^{\omega}}%omega的无穷子集全体
\newcommand{\func}[3]{#1 \colon #2 \to #3} %函数#1: #2 -> #3
\newcommand{\set}[2]{\{#1 \colon #2\}} %集合{#1: #2}
\newcommand{\cantor}{2^\omega} %The Cantor space
\newcommand{\fullBinaryTree}{2^{<\omega}} %The infinite full binary tree
\newcommand{\proj}{\operatorname{proj}} %投影
\newcommand{\pow}[1]{\operatorname{Pow}(#1)} %表示幂集
\newcommand{\symdiff}{\mathbin{\triangle}} %表示对称差
\renewcommand{\restriction}{\mathord{\upharpoonright}} %限制

\newcommand{\CI}{\mathcal{I}}
\newcommand{\CJ}{\mathcal{J}}

\newcommand{\CC}{\mathcal{C}}

\newcommand{\CE}{\mathcal{E}}

\newcommand{\CW}{\mathcal{W}}

\newcommand{\SD}{\mathscr{D}}

\newcommand{\code}[2]{\langle #1,#2 \rangle} %bijection from \omega \times \omega to \omega
\newcommand{\fin}{\mathsf{Fin}}

\makeatletter
\def\ps@pprintTitle{%
 \let\@oddhead\@empty
 \let\@evenhead\@empty
 \def\@oddfoot{\centerline{\thepage}}%
 \let\@evenfoot\@oddfoot}
\makeatother

\allowdisplaybreaks
\begin{document}

\begin{frontmatter}

\title{Closed maximal eventually different families for $\boldsymbol{\Sigma}^0_3$ and weakly Farah ideals}

\author[1]{Jialiang He}
\ead{jialianghe@scu.edu.cn}

\author[2]{Yi Zhou}
\ead{zhouyi.math@outlook.com}

\author[1]{Jiaheng Zhuang\corref{cor}}
\ead{jiahengzhuang@outlook.com}

\cortext[cor]{Corresponding author.}
\address[1]{School of Mathematics, Sichuan University, Chengdu 610064, P.R.China}
\address[2]{School of Sciences, Southwest Petroleum University, Chengdu, 610500, P.R.China}
\begin{abstract}
Motivated by the existence of closed maximal $\mathcal{I}$-eventually different ($\mathcal{I}$-MED) families, we study uniform combinatorial properties of ideals.
For a fixed ideal, a property is uniform if its witnesses can be chosen by a single Borel map.
We prove that every $\boldsymbol{\Sigma}^0_3$ ideal is uniformly weakly $P^+$.
We also prove that the classes of Farah and weakly Farah ideals coincide, answering a question of Hru\v{s}\'ak and Meza-Alc\'antara.
We then show that every weakly Farah ideal is uniformly weakly Ramsey.
As an application, we obtain closed $\mathcal{I}$-MED families for $\boldsymbol{\Sigma}^0_3$ ideals and weakly Farah ideals.
\end{abstract}

\begin{keyword}
ideals \sep uniformly weakly $P^+$ \sep maximal eventually different families for ideals \sep uniformly weakly Ramsey \sep weakly Farah ideals

\MSC[2020] 03E15 \sep 03E05
%\MSC[2020]  03E15 \sep 03B52
\end{keyword}
\end{frontmatter}

\section{Introduction}

A family $\CE \subseteq \omega^\omega$ is \emph{eventually different} if any two distinct members agree at only finitely many coordinates, and \emph{maximal eventually different}, or \emph{MED}, if it is maximal with respect to inclusion among such families.
Like other maximal combinatorial objects, MED families can be constructed using the axiom of choice, and it is natural to ask how definable they can be.
Horowitz and Shelah \cite{HS2024} showed that there is a Borel MED family, and Schrittesser \cite{S2018} showed that there is even a closed one.

Throughout, an \emph{ideal} means an ideal on $\omega$.
He, Luo, Schrittesser and Zhang \cite{HLSZ} introduced the following relativization of MED families to an ideal $\CI$.
A family $\CE \subseteq \omega^\omega$ is an \emph{$\CI$-MED family} if $\set{n}{f(n) = g(n)} \in \CI$ for all distinct $f, g \in \CE$, and every $h \in \omega^\omega$ agrees with some $f \in \CE$ on an $\CI$-positive set. 
For $\CI = \fin$, these are exactly the MED families. They proved that every $F_\sigma$ ideal admits a closed $\CI$-MED family \cite[Corollary~6.3]{HLSZ}.
Our aim is to extend this existence result beyond
$F_\sigma$ ideals.
For Borel ideals, the uniformly weakly Ramsey property
suffices for this purpose
\cite[Theorems~4.8 and~6.2]{HLSZ}.
Our first goal is to apply this criterion to $\boldsymbol{\Sigma}^0_3$ ideals.

\subsection*{The weakly \texorpdfstring{$P^+$}{P+}-property for \texorpdfstring{$\boldsymbol{\Sigma}^0_3$}{Sigma03} ideals}
Balcerzak, \text{G\l{}\k{a}b}
and Leonetti proved that every $\boldsymbol{\Sigma}^0_3$ ideal is $P^+$ \cite[Corollary~2.7]{BGL2026}.
Their proof gives a positive pseudointersection for each decreasing sequence of positive sets, but does not provide a Borel map choosing one. 

For our application to closed $\CI$-MED families,
it suffices to obtain a Borel witness for the weakly
$P^+$ property, a weakening of $P^+$.
Indeed, every Borel uniformly weakly $P^+$ ideal is
uniformly weakly Ramsey \cite[Lemma~5.2]{HLSZ}.
We prove the following.

\begin{mainthm}\label{mthm:Sig03}
Every $\boldsymbol{\Sigma}^0_3$ ideal is uniformly weakly $P^+$.
\end{mainthm}

To prove Theorem~\ref{mthm:Sig03}, we first establish
a uniform version of the Jalali-Naini--Talagrand theorem
\cite{JN1976,T1980} for Borel ideals.
We then combine this result with a continuous reduction and the contraction mapping theorem to construct the required Borel witness.

\subsection*{Weakly Farah ideals}
After treating $\boldsymbol{\Sigma}^0_3$ ideals, we turn to $\boldsymbol{\Pi}^0_3$ ideals.
Here the main difficulty is finding representations that can be used to construct Borel witnesses.
Solecki's theorem provides a representation of every analytic $P$-ideal by a lower semicontinuous submeasure \cite{S1999}.
For general $\boldsymbol{\Pi}^0_3$ ideals, Farah asked whether every such ideal is Farah \cite[Question~15.2]{Farah}, or equivalently, whether it is a countable intersection of hereditary $F_\sigma$ sets closed under finite changes \cite[Proposition~3.6]{HM2011}.

Dropping the requirement of closure under finite changes gives the class of \emph{weakly Farah ideals}, introduced by Hru\v{s}\'ak and Meza-Alc\'antara \cite[Definition~3.11]{HM2011}.
This class contains all analytic $P$-ideals and retains the hereditary structure that we use to obtain Borel witnesses.
We prove that every weakly Farah ideal is Farah, answering a question of Hru\v{s}\'ak and Meza-Alc\'antara \cite[Section~5, Question~(3)]{HM2011}.
We then show that every weakly Farah ideal is uniformly weakly Ramsey. Here  the hereditary $F_\sigma$ representation allows us to choose Ramsey witnesses in a Borel way.

\begin{mainthm}\label{mthm:Farah}
Every weakly Farah ideal is Farah and uniformly weakly Ramsey.
\end{mainthm}

\subsection*{Closed \texorpdfstring{$\CI$}{I}-MED families}
The two uniformity results yield the following corollary.

\begin{maincor}\label{mcor:MED}
Let $\CI$ be a $\boldsymbol{\Sigma}^0_3$ ideal or a weakly Farah ideal.
Then there is a closed $\CI$-MED family.
\end{maincor}

\subsection*{Organization}
Section~\ref{sec:Preliminaries} fixes notation and collects the facts from descriptive set theory that we use.
In Section~\ref{sec:uniform-p+}, we prove that every $\boldsymbol{\Sigma}^0_3$ ideal is uniformly weakly $P^+$.
Section~\ref{sec:uwR} shows that every weakly Farah ideal is Farah and uniformly weakly Ramsey.
Finally, Section~\ref{sec:applications} gives the applications to closed $\CI$-MED families.

\section{Preliminaries}\label{sec:Preliminaries}

Our terminology and notation in descriptive set theory follow \cite{K1995}.

For $s, t \in \fullBinaryTree$ we write $\abs{s}$ for the length of $s$, and $s \preceq t$ if $t$ extends $s$. 
We let $[s] = \set{x \in \cantor}{s \preceq x}$ be the corresponding basic open set.

Fix a bijection $\func{\code{\cdot}{\cdot}}{\omega\times\omega}{\omega}$.
For $n \in \omega$, we let $(n)_0$ and $(n)_1$ be the unique numbers with $n = \code{(n)_0}{(n)_1}$.

We identify $\pow{\omega}$ with $\cantor$ via characteristic functions and equip it with the induced Polish topology.
We therefore also use set-theoretic notation, such as $x\subseteq y$, $x\cap y$, and $x\setminus y$, for elements $x,y\in\cantor$.

A family $\CC\subseteq\cantor$ is \emph{hereditary} if $B\subseteq A\in \CC$ implies $B\in \CC$.
An \emph{ideal}  is a hereditary family $\CI\subseteq\pow{\omega}$ that contains every finite subset of $\omega$, is closed under finite unions, and does not contain $\omega$.
We write $\CI^+=\pow{\omega}\setminus\CI$. 
For $B \in \CI^+$, let $\CI\restriction{B} = \set{A \cap B}{A \in \CI}$.
Let $\fin$ denote the set of all finite subsets of $\omega$.

\subsection{Maximal eventually different families}
Let $\CI$ be an ideal. A family $\CE\subseteq\omega^\omega$ is an \emph{$\CI$-maximal eventually different family}, or an
\emph{$\CI$-MED family}, if
\[
[f=g]=\set{n\in\omega}{f(n)=g(n)}\in\CI
\]
for all distinct $f,g\in\CE$, and for every $h\in\omega^\omega$, there is $f\in\CE$ such that $[h=f]\in\CI^+$.

\subsection{Descriptive set theory} 
We recall two results from descriptive set theory for later use.

The first is due to Fremlin.

\begin{theorem}[{\cite[Lemma~1.7]{BJ1994}}]\label{thm:Fremlin}
Let $E\subseteq\cantor\times\cantor$ be Borel with every section $E_B$ meager.
Then there are Borel sets $F_n\subseteq\cantor\times\cantor$ with closed nowhere dense sections such that
\[
E\subseteq\bigcup_n F_n.
\]
\end{theorem}

The second is a Borel uniformization theorem due to
Arsenin and Kunugui.

\begin{theorem}[{\cite[Theorem~35.46]{K1995}}]
\label{thm:compact-section}
Let $X$ be a standard Borel space, $Y$ a nonempty
Polish space, and $A\subseteq X\times Y$ a Borel set
whose sections $A_x$ are $\sigma$-compact.
Then there is a Borel map $\func{p}{X}{Y}$ such that
\[
p(x)\in A_x
\qquad(\forall x\in\proj_X(A)).
\]
In particular, $\proj_X(A)$ is Borel.
\end{theorem}

\subsection{Weakly \texorpdfstring{$P^+$}{P+} and weakly Ramsey ideals}\label{sec:weakly-P+--weakly-Ramsey}

We recall the definition of the  weakly $P^+$ property and weakly Ramsey property.

Let $\CI$ be an ideal and let $\vec{A}=(A_n)_n$ be a
decreasing sequence of $\CI$-positive sets.
\begin{itemize}
\item  A set $A\subseteq\omega$ is a
    \emph{pseudointersection} of $\vec{A}$ if
    $A\setminus A_n\in \fin$ for every $n\in\omega$.
    \item A set $A\subseteq\omega$ is an
    \emph{$\CI$-pseudointersection} of $\vec{A}$ if
    $A\setminus A_n\in\CI$ for every $n\in\omega$.
    \item The sequence $\vec{A}$ is
    \emph{$\CI$-strictly decreasing} if
    $A_n\setminus A_{n+1}\in\CI^+$ for every $n\in\omega$.
\end{itemize}
We say that $\CI$ is \emph{uniformly weakly $P^+$} if
there is a Borel map
\[
\func{\Pi}{(\cantor)^\omega}{\cantor}
\]
such that $\Pi(\vec{A})$ is an $\CI$-positive
$\CI$-pseudointersection of $\vec{A}$ for every decreasing
sequence $\vec{A}$ of $\CI$-positive sets.

For Borel ideals, it suffices to require this only for
$\CI$-strictly decreasing sequences.
To see this, suppose that $\CI$ is Borel and that $\Pi$
satisfies this restricted requirement.
Given a decreasing sequence $\vec{A}=(A_n)_n$ of
$\CI$-positive sets, put
\[
S(\vec{A})=
\set{n\in\omega}{A_n\setminus A_{n+1}\in\CI^+}.
\]
If $S(\vec{A})$ is finite, let $n_0$ be the least integer
such that $A_n\setminus A_{n+1}\in\CI$ for all $n\geq n_0$.
Then $A_{n_0}$ is an $\CI$-positive
$\CI$-pseudointersection of $\vec{A}$, so we use $A_{n_0}$
as the output.

If $S(\vec{A})$ is infinite, let $(n_k)_k$ be its
increasing enumeration.
Then $(A_{n_k})_k$ is $\CI$-strictly decreasing and has
the same $\CI$-pseudointersections as $\vec{A}$.
Thus we use $\Pi((A_{n_k})_k)$ as the output.

Since $\CI$ is Borel, the case distinction and all
choices above are Borel.
Setting the output equal to $\varnothing$ outside the
Borel set of decreasing sequences of $\CI$-positive
sets gives the required Borel map on $(\cantor)^\omega$.

We identify $2^{[\omega]^2}$ with $\cantor$ via a fixed bijection between $[\omega]^2$ and $\omega$.
For a coloring $\func{c}{[\omega]^2}{2}$ and $x<y$, we write $c(x,y)$ for $c(\{x,y\})$.

A coloring $c$ is \emph{subadditive} if
\[
c(x,y) \leq c(x,z) + c(y,z)
\qquad (x < y < z).
\]
Subadditivity is equivalent to the following condition:
\begin{equation}\label{eq:common-right-endpoint}
c(x,z)=c(y,z)=0 \ \implies\ c(x,y)=0
\qquad (x<y<z).
\end{equation}
A set $H \subseteq \omega$ is \emph{$0$-homogeneous} for $c$ if $c \restriction [H]^2 \equiv 0$.
It is \emph{nowhere $0$-homogeneous} for $c$ if every $\CI$-positive set $B \subseteq H$ contains $x<y$ with $c(x,y)=1$; equivalently, every $0$-homogeneous subset of $H$ belongs to $\CI$.

An ideal $\CI$ is \emph{uniformly weakly Ramsey} if there is a Borel map
\[
\func{H}{\INF\times 2^{[\omega]^2}}{\INF}
\]
such that, for every $A\in\CI^+$ and every subadditive coloring
$\func{c}{[\omega]^2}{2}$, the set $H(A,c)$ is an
$\CI$-positive subset of $A$ that is either $0$-homogeneous
or nowhere $0$-homogeneous for $c$.

\subsection{Weakly Farah ideals and Farah ideals}

We recall the definitions of Farah and weakly Farah ideals.

An ideal $\CI$ is \emph{Farah} if there is a sequence $(C_n)_{n\in\omega}$ of compact hereditary families such that
\begin{equation}\label{eq:farah-representation}
A\in\CI
\quad\iff\quad
(\forall n \in \omega)\,(\exists m \in \omega)\ A\setminus m\in C_n.
\end{equation}
It is \emph{weakly Farah} if $\CI=\bigcap_n H_n$ for some hereditary $F_\sigma$ sets $H_n$ \cite[Definition~3.11]{HM2011}.
A family $\CC \subseteq \pow{\omega}$ is \emph{closed under finite changes} if $A \in \CC$ and $\abs{A \symdiff A'} < \infty$ imply $A' \in \CC$.

\begin{theorem}[{\cite[Proposition~3.6]{HM2011}}]\label{thm:Farah-ideal}
An ideal $\CI$ is Farah if and only if $\CI=\bigcap_n F_n$ for some sequence $(F_n)_{n\in\omega}$ of hereditary $F_\sigma$ sets closed under finite changes.
\end{theorem}

In particular, every Farah ideal is weakly Farah, and every weakly Farah ideal is $\boldsymbol{\Pi}^0_3$.

\section{\texorpdfstring{$\boldsymbol{\Sigma}^0_3$}{Sigma03} ideals}\label{sec:uniform-p+}

In this section, we prove that every
$\boldsymbol{\Sigma}^0_3$ ideal is uniformly weakly $P^+$.
For the reader's convenience, we first recall why such ideals are weakly $P^+$, following the argument in
\cite[Proposition~3.2(ix)]{FKL2026}.

Let $\CI$ be a $\boldsymbol{\Sigma}^0_3$ ideal.
Given an $\CI$-strictly decreasing sequence
$\vec{A}=(A_n)_n$ of $\CI$-positive sets, define
\begin{equation}\label{eq:Sigma3-partition}
B_0=(\omega\setminus A_1)\cup\bigcap_k A_k,
\qquad
B_k=A_k\setminus A_{k+1}\quad(k\geq1).
\end{equation}
These sets partition $\omega$ into $\CI$-positive sets:
$B_0$ contains $A_0\setminus A_1$, and each $B_k$ with
$k\geq1$ is positive by the strict decrease of $\vec{A}$.
Since $\omega\setminus A_n\subseteq\bigcup_{k\leq n}B_k$
for every $n\in\omega$, we have
\begin{equation}\label{eq:part-prop}
(\forall k)\ \abs{B\cap B_k}<\infty
\quad\implies\quad
(\forall n)\ \abs{B\setminus A_n}<\infty
\end{equation}
for every $B\subseteq\omega$.
Thus any set meeting each $B_k$ in a finite set
is a pseudointersection of $\vec{A}$.
The Jalali-Naini--Talagrand theorem then partitions each
$B_n$ into nonempty finite blocks $(e_{n,m})_{m\in\omega}$
such that every infinite union of these blocks is
$\CI$-positive.

If $\vec{A}$ had no positive pseudointersection, then,
for every $X\subseteq\omega \times \omega$,
\[
\bigcup_{(n,m)\in X}e_{n,m}\in\CI
\quad\iff\quad
X\in\varnothing\times\fin,
\]
where
\[
\varnothing\times\fin
= \set{X\subseteq\omega \times \omega}{(\forall n\in\omega)\ \abs{X_n}<\infty},
\qquad
X_n=\set{m\in\omega}{(n,m)\in X}.
\]
Indeed, if every $X_n$ is finite, the union is almost
contained in every $A_k$ and hence belongs to $\CI$
by assumption.
If some $X_n$ is infinite, the union contains an infinite
union of blocks from $B_n$ and is therefore $\CI$-positive.

The map
\[
F:2^{\omega \times \omega}\longrightarrow 2^\omega,
\qquad
F(X)=\bigcup_{(n,m)\in X}e_{n,m},
\]
is continuous, and the equivalence above gives
\[
\varnothing\times\fin=F^{-1}(\CI).
\]
Since $\CI$ is $\boldsymbol{\Sigma}^0_3$, its continuous
preimage $\varnothing\times\fin$ would also be
$\boldsymbol{\Sigma}^0_3$, contradicting the
$\boldsymbol{\Pi}^0_3$-completeness of
$\varnothing\times\fin$.

To make this argument uniform, we first prove a uniform
version of the Jalali-Naini--Talagrand theorem, which allows
the finite blocks $e_{n,m}$ to be chosen in a Borel way
from $\vec{A}$.
We then use a fixed continuous reduction to define a
Borel family of contractions. Their unique fixed points
give the required positive pseudointersections.

\subsection{A uniform Jalali-Naini--Talagrand theorem}\label{subsec:JNT}
Recall that an ideal $\CJ$ is \emph{Rudin--Blass reducible to}
an ideal $\CI$, written $\CJ\leq_{RB}\CI$, if there is a
finite-to-one map $\func{f}{\omega}{\omega}$ such that,
for every $X\subseteq\omega$,
\[
X\in\CJ \quad\iff\quad f^{-1}[X]\in\CI.
\]

Mathias \cite{M1973} proved that $\fin\leq_{RB}\CI$
for every analytic ideal $\CI$.
Jalali-Naini \cite{JN1976} and Talagrand \cite{T1980}
independently characterized the ideals satisfying this
reducibility condition in terms of Baire category.

\begin{theorem}[Jalali-Naini, Talagrand]\label{thm:JNT}
For an ideal $\CI$, the following are equivalent:
\begin{enumerate}
    \item $\CI$ has the Baire property;
    \item $\CI$ is meager;
    \item $\fin\leq_{RB}\CI$.
\end{enumerate}
\end{theorem}

If $\CI$ is Borel, then for every $B\in\CI^+$,
the restriction $\CI\restriction B$ is Borel in $2^B$
and hence has the Baire property.
Applying Theorem~\ref{thm:JNT} to $\CI\restriction B$
gives a partition $(e_n)_{n\in\omega}$ of $B$ into
nonempty finite sets such that
\[
\bigcup_{n\in C}e_n\in\CI^+
\qquad (\forall C\in\INF).
\]
We next show that such a partition can be chosen
in a Borel way from $B\in\CI^+$.

\begin{lemma}\label{lem:Borel-JNT}
Let $\CI$ be a Borel ideal.
There is a Borel map
\[
\func{\Lambda}{\CI^+}
{([\omega]^{<\omega}\setminus\{\varnothing\})^\omega}
\]
such that, for every $B\in\CI^+$, the sequence
$\Lambda(B)=(e_n(B))_{n\in\omega}$ satisfies:
\begin{itemize}
    \item $\max e_n(B)<\min e_{n+1}(B)$ for every $n\in\omega$;
    \item $\bigcup_{n\in\omega}e_n(B)=B$;
    \item $\bigcup_{n\in C}e_n(B)\in\CI^+$ for every $C\in\INF$.
\end{itemize}
\end{lemma}

\begin{proof} We follow Talagrand's construction \cite{T1980} and verify
that it yields a Borel map.
We first represent the restrictions of $\CI$ on the fixed
space $\cantor$.
Define the continuous map
$\func{f}{\cantor\times\cantor}{\cantor}$ by
\[
k\in f(B,x)
\quad\iff\quad
k\in B \ \land\ |B\cap k|\in x.
\]
For infinite $B$, the map $f(B,\cdot)$ is the homeomorphism
from $\cantor$ onto $\pow{B}$ induced by the increasing
enumeration of $B$.

Consider the Borel set
\[
E=
\set{(B,x)\in\cantor\times\cantor}
{B\in\CI^+ \text{ and } f(B,x)\in\CI}.
\]
If $B\in\CI^+$, then $E_B$ is a Borel ideal,
so it is meager by Theorem~\ref{thm:JNT}.
Otherwise, $E_B=\varnothing$.
By Theorem~\ref{thm:Fremlin}, there are Borel sets
$F_n\subseteq\cantor\times\cantor$ with closed nowhere
dense sections such that $E\subseteq\bigcup_n F_n$.
Replacing $F_n$ by $\bigcup_{j\leq n}F_j$, we may assume
that $F_n\subseteq F_{n+1}$.

We next choose the endpoints of the finite blocks.
For fixed $n\in\omega$ and $t\in2^{<\omega}$, the set
\[
\set{B\in\cantor}{[t]\cap(F_n)_B\neq\varnothing}
\]
is the projection of
$F_n\cap(\cantor\times[t])$ onto the first coordinate.
Its sections are compact, so this projection is Borel
by Theorem~\ref{thm:compact-section}.
Consequently, the condition
$[t]\cap(F_n)_B=\varnothing$ is Borel in $B$.

Set $m_0(B)=0$.
Given $m_n(B)$, let $m_{n+1}(B)$ be the least
$\ell>m_n(B)$ such that
\begin{equation}\label{eq:JNT-extension}
(\forall s\in2^{m_n(B)})(\exists t\in2^\ell)\,
\bigl(s\preceq t \ \land\
[t]\cap(F_n)_B=\varnothing\bigr).
\end{equation}
Such an $\ell$ exists because $(F_n)_B$ is closed nowhere
dense and $2^{m_n(B)}$ is finite.
The preceding observation makes
\eqref{eq:JNT-extension} a Borel condition on $(B,\ell)$,
provided that $m_n$ is Borel.
Thus each map $m_n:\cantor\to\omega$ is Borel by induction.

For $B\in\CI^+$, set
\[
I_n(B)=[m_n(B),m_{n+1}(B)),
\qquad
e_n(B)=f(B,I_n(B)),
\]
and define $\Lambda(B)=(e_n(B))_{n\in\omega}$.
The intervals $I_n(B)$ partition $\omega$ into nonempty
finite sets.
Hence the sets $e_n(B)$ partition $B$ into nonempty
finite sets and satisfy
\[
\max e_n(B)<\min e_{n+1}(B).
\]
Moreover, $\Lambda$ is Borel, since
\[
k\in e_n(B)
\quad\iff\quad
k\in B \ \land\
m_n(B)\leq |B\cap k|<m_{n+1}(B).
\]

It remains to prove that every infinite union of the
blocks is $\CI$-positive.
Fix $B\in\CI^+$ and $C\in\INF$, and set
\[
X=\bigcup_{n\in C}I_n(B).
\]
Since $E_B$ is hereditary, it suffices to construct
$z\subseteq X$ with $z\notin E_B$.

Construct $z$ successively on the intervals $I_n(B)$.
If $n\notin C$, set $z$ equal to $0$ on $I_n(B)$.
If $n\in C$, use \eqref{eq:JNT-extension} to extend
$z\restriction m_n(B)$ to a sequence $t$ of length
$m_{n+1}(B)$ such that
$[t]\cap(F_n)_B=\varnothing$.

Then $z\subseteq X$ and $z\notin(F_n)_B$ for every
$n\in C$.
Since $C$ is infinite and the sets $F_n$ are increasing,
$z$ lies outside every $(F_n)_B$, and hence outside $E_B$.
It follows that $X\notin E_B$, so
\[
\bigcup_{n\in C}e_n(B)=f(B,X)\in\CI^+,
\]
as required.
\end{proof}

Recall that, for $q\geq 0$, a self-map $T$ of a metric space $(X,d)$ is \emph{$q$-Lipschitz} if
\[
d(T(x),T(y))\leq q\,d(x,y) \qquad(\forall x,y\in X).
\]
It is a \emph{contraction} if it is $q$-Lipschitz for some $q\in[0,1)$.
The contraction mapping theorem guarantees that every contraction on a nonempty complete metric space has a unique fixed point.
The following lemma gives Borel dependence of this fixed point on the parameter.

\begin{lemma}\label{lemma:Borel-fixed-points}
Let $D$ be a standard Borel space and $X$ a nonempty complete separable metric space.
Suppose that $\func{F}{D\times X}{X}$ is Borel and, for every $x\in D$, the map $\func{F_x}{X}{X}$ defined by $F_x(z)=F(x,z)$ is a contraction.
Then there is a unique Borel map $\func{p}{D}{X}$ such that
\[
p(x)=F_x(p(x))
\qquad(x\in D).
\]
\end{lemma}

\begin{proof}
Consider the Borel set
\[
E=\set{(x,z)\in D\times X}{z=F(x,z)}.
\]
By the contraction mapping theorem, each section $E_x$
is a singleton.
Theorem~\ref{thm:compact-section} therefore gives a
Borel map $\func{p}{D}{X}$ such that
$p(x)\in E_x$ for every $x\in D$.
Thus $p(x)=F_x(p(x))$, and uniqueness of the fixed
point in each section gives uniqueness of $p$.
\end{proof}

\subsection{Proof of the main theorem}
By Subsection~\ref{sec:weakly-P+--weakly-Ramsey},
it suffices to choose Borelly an $\CI$-positive
$\CI$-pseudointersection for each $\CI$-strictly
decreasing sequence of $\CI$-positive sets.
The construction below yields an $\CI$-positive
pseudointersection for every such sequence.

\begin{theorem}\label{thm:Sig03-uniform-p+}
Every $\boldsymbol{\Sigma}^0_3$ ideal is uniformly weakly $P^+$.
\end{theorem}
\begin{proof}
Let $\CI$ be a $\boldsymbol{\Sigma}^0_3$ ideal, and let
\[
\SD=\set{\vec{A}\in(\cantor)^\omega}
{\vec{A}\text{ is an $\CI$-strictly decreasing sequence
of $\CI$-positive sets}}.
\]
The set $\SD$ is Borel.
As discussed in Subsection~\ref{sec:weakly-P+--weakly-Ramsey}, it suffices to construct the required Borel map on $\SD$.

For $\vec{A}=(A_n)_n\in\SD$, let $(B_k)_k$ be the
partition defined in \eqref{eq:Sigma3-partition} and $\Lambda$ be as in Lemma~\ref{lem:Borel-JNT}.
Write $\Lambda(B_k)=(e_{k,m})_{m\in\omega}$, so that
$(e_{k,m})_m$ partitions $B_k$ into successive nonempty
finite sets and
\begin{equation}\label{eq:inf-block-positive}
\bigcup_{m\in C}e_{k,m}\in\CI^+
\qquad(C\in\INF).
\end{equation}
The sets $B_k$ and $e_{k,m}$ depend Borelly on $\vec{A}$;
we suppress this dependence in the notation.

Using the pairing function fixed in Section~\ref{sec:Preliminaries}, define the ideal
\begin{equation}\label{eq:J-definition}
\CJ=\set{x\subseteq\omega}
{(\forall k\in\omega)\ \abs{\set{m\in\omega}{\code{k}{m}\in x}}<\infty}.
\end{equation}
Thus $\CJ$ is the copy of $\varnothing\times\fin$ on $\omega$ induced by the pairing function, and it is $\boldsymbol{\Pi}^0_3$-complete.
Since $\CI^+\in\boldsymbol{\Pi}^0_3$, fix a continuous map $\func{g}{\cantor}{\cantor}$ such that 
\begin{equation}\label{eq:Sigma3-reduction}
Z\in\CI^+
\quad\iff\quad
g(Z)\in \CJ.
\end{equation}
By continuity and compactness, choose strictly increasing
integers $r_j\in\omega$ such that $g(z)(j)$ depends only
on $z\restriction r_j$.
Set
\[
\epsilon_j(\vec{A})=e_{(j)_0,r_j}.
\]
These finite sets are pairwise disjoint.
Since each $(e_{k,m})_m$ consists of successive nonempty
blocks, $\min e_{k,m}\geq m$, and hence $\min\epsilon_j(\vec{A})\geq r_j$.

Define
\[
\func{F}{\SD\times\cantor}{\cantor},
\qquad
F(\vec{A},z)=\bigcup_{j\in g(z)}\epsilon_j(\vec{A}).
\]
We verify the hypotheses of
Lemma~\ref{lemma:Borel-fixed-points}.
The map $F$ is Borel, since
\[
k\in F(\vec{A},z)
\quad\iff\quad
(\exists j)\,
\bigl(k\in\epsilon_j(\vec{A})\ \land\ g(z)(j)=1\bigr).
\]

We next show that each section $F_{\vec{A}}$ is
$\tfrac{1}{2}$-Lipschitz for the usual complete metric
\[
d(z,z')=2^{-\min\{n:z(n)\ne z'(n)\}}
\quad(z\ne z'),
\qquad d(z,z)=0.
\]
Fix $\vec{A}\in\SD$ and distinct $z,z'\in\cantor$, and let
$n=\min\{k:z(k)\ne z'(k)\}$.
If $k\leq n$ and $k\in\epsilon_j(\vec{A})$, then
$r_j\leq k\leq n$, so $g(z)(j)=g(z')(j)$.
Outside the selected blocks, both outputs are zero.
Thus
\[
F_{\vec{A}}(z)\restriction(n+1)
=
F_{\vec{A}}(z')\restriction(n+1),
\]
and consequently
\[
d\bigl(F_{\vec{A}}(z),F_{\vec{A}}(z')\bigr)
\leq 2^{-(n+1)}
=\tfrac12d(z,z').
\]

By Lemma~\ref{lemma:Borel-fixed-points}, there is
a Borel map $\func{\Pi}{\SD}{\cantor}$ such that
\begin{equation}\label{eq:Sigma3-fixed-point}
\Pi(\vec{A})
=
F_{\vec{A}}(\Pi(\vec{A}))
=
\bigcup_{j\in g(\Pi(\vec{A}))}\epsilon_j(\vec{A}).
\end{equation}

Fix $\vec{A}\in\SD$, and write
\[
x=g(\Pi(\vec{A})),
\qquad
X_k=\set{m\in\omega}{\code{k}{m}\in x}.
\]
We verify that
\begin{enumerate}[label=(G\arabic*),ref=G\arabic*]
    \item\label{item:G1}
    $\Pi(\vec{A})\in\CI^+$;
    \item\label{item:G2}
    $\abs{\Pi(\vec{A})\cap B_k}<\infty$ for every $k$.
\end{enumerate}

For~\eqref{item:G1}, suppose that $\Pi(\vec{A})\in\CI$.
By \eqref{eq:Sigma3-reduction}, $x\notin \CJ$, so some
$X_{k_0}$ is infinite.
Since the integers $r_j$ are distinct,
\eqref{eq:Sigma3-fixed-point} and
\eqref{eq:inf-block-positive} give
\[
\Pi(\vec{A})
\supseteq
\bigcup_{m\in X_{k_0}}e_{k_0,r_{\code{k_0}{m}}}
\in\CI^+,
\]
contradicting $\Pi(\vec{A})\in\CI$.

For~\eqref{item:G2}, \eqref{item:G1} and
\eqref{eq:Sigma3-reduction} imply that $x\in \CJ$.
Thus every $X_k$ is finite.
Since the sets $B_k$ are pairwise disjoint,
\eqref{eq:Sigma3-fixed-point} gives
\[
\Pi(\vec{A})\cap B_k
=
\bigcup_{m\in X_k}e_{k,r_{\code{k}{m}}},
\]
which is finite.
By \eqref{eq:part-prop}, $\Pi(\vec{A})$ is therefore
an $\CI$-positive pseudointersection of $\vec{A}$,
as required.
\end{proof}

A proof of the stronger statement that every
$\boldsymbol{\Sigma}^0_3$ ideal is uniformly $P^+$
will be presented in our separate paper devoted to uniform
combinatorial properties of ideals.
The present paper establishes the uniformly weakly $P^+$
property, which suffices for our applications to closed
$\CI$-MED families.

\section{\texorpdfstring{$\boldsymbol{\Pi}^0_3$}{Pi03} ideals and weakly Farah  ideals}\label{sec:uwR}

The preceding section establishes the uniform weakly $P^+$
property for $\boldsymbol{\Sigma}^0_3$ ideals.
We now turn to $\boldsymbol{\Pi}^0_3$ ideals.
It is not known whether every such ideal can be written
as a countable intersection of hereditary $F_\sigma$ sets.
To our knowledge, all known examples admit such a
representation and hence are weakly Farah.
We show that this representation allows us to construct
Borel Ramsey witnesses.
Our main result is as follows.

\begin{theorem}\label{thm:weakly-farah-ideals-uwr}
Every weakly Farah ideal is uniformly weakly Ramsey.
\end{theorem}

To prove this theorem, we first show that the hereditary
$F_\sigma$ sets in the definition of a weakly Farah ideal
can be chosen to be closed under finite changes.
Equivalently, every weakly Farah ideal is Farah.
This answers the second part of
\cite[Section~5, Question~(3)]{HM2011}.

\begin{prop}\label{prop:weakly-farah-is-farah}
An ideal is weakly Farah if and only if it is Farah.
\end{prop}
\begin{proof}
Every Farah ideal is weakly Farah by Theorem~\ref{thm:Farah-ideal}, so only the forward implication requires proof.
Write $\CI = \bigcap_n H_n$, where each $H_n$ is hereditary and $F_\sigma$.
By Theorem~\ref{thm:Farah-ideal} it suffices to represent $\CI$ as a countable intersection of hereditary $F_\sigma$ sets closed under finite changes.

For $n\in\omega$ and $F\in[\omega]^{<\omega}$, set
\[
T_{n,F} =\set{A\subseteq\omega}{(\exists k \in \omega)\,F\cup(A\setminus k)\in H_n}.
\]
Each $T_{n,F}$ is hereditary since $H_n$ is hereditary.
It is $F_\sigma$ since the maps $A\mapsto F\cup(A\setminus k)$ are continuous.
To check closure under finite changes, let $A\in T_{n,F}$ and suppose that $A\symdiff A'$ is finite. 
Choose $k$ such that $F\cup(A\setminus k)\in H_n$, and then choose $k'\geq k$ such that
$A\symdiff A'\subseteq k'$. Thus
\[
F\cup(A'\setminus k') =F\cup(A\setminus k') \subseteq F\cup(A\setminus k).
\]
Since $H_n$ is hereditary, $F\cup(A'\setminus k') \in H_n$, so $A'\in T_{n,F}$.

We claim that
\[
\CI=\bigcap_{n\in\omega,\,F\in[\omega]^{<\omega}}T_{n,F}.
\]
If $A \in \CI$, then $F \cup A \in \CI \subseteq H_n$ for every finite $F$; taking $k = 0$ gives $A \in T_{n,F}$.

For the converse, let $A$ belong to the right-hand side.
Starting with $m_0=0$, recursively define $m_0<m_1<\cdots$ as follows.
Given $m_0,\ldots,m_{2n}$, set
\[
B_n=A\cap\bigcup_{j<n}[m_{2j+1},m_{2j+2}),
\]
and use $A \in T_{n,B_n}$ to choose $m_{2n+1} > m_{2n}$ with $B_n \cup (A \setminus m_{2n+1}) \in H_n$.
Next set
\[
D_n=A\cap\bigcup_{j\leq n}[m_{2j},m_{2j+1}),
\]
and use $A \in T_{n,D_n}$ to choose $m_{2n+2} > m_{2n+1}$ with $D_n \cup (A \setminus m_{2n+2}) \in H_n$.

Finally set
\[
B=A\cap\bigcup_n[m_{2n+1},m_{2n+2}), \qquad D=A\cap\bigcup_n[m_{2n},m_{2n+1}).
\]
For every $n$ we have $B\subseteq B_n\cup(A\setminus m_{2n+1})$ and $D\subseteq D_n\cup(A\setminus m_{2n+2})$, so $B, D \in H_n$ because $H_n$ is hereditary.
Hence $B, D \in \bigcap_n H_n = \CI$, and $A = B \cup D \in \CI$.
This proves the claim.

Since each $T_{n,F}$ is a hereditary $F_\sigma$ set
closed under finite changes, the claim shows that
$\CI$ is Farah.
\end{proof}

We now use the representation~\eqref{eq:farah-representation}
to prove Theorem~\ref{thm:weakly-farah-ideals-uwr}.

\begin{proof}[Proof of Theorem~\ref{thm:weakly-farah-ideals-uwr}]
Fix a weakly Farah ideal $\CI$. 
By Proposition~\ref{prop:weakly-farah-is-farah}, choose compact hereditary sets $(C_n)_n$ satisfying \eqref{eq:farah-representation}.
Let
\[
\CW=
\set{(A,c)\in\INF\times2^{[\omega]^2}}{A\in\CI^+\text{ and }c\text{ is subadditive}}.
\]
The set $\CW$ is Borel. 
We define $H$ on $\CW$, and set $H(A,c) = A$ for $(A,c) \notin \CW$.

For each $n\in\omega$, let $R_n(A,c)$ be the Borel condition
\[
(\forall m \in \omega)(\exists Q \in [A \setminus m]^{<\omega})\
\bigl(c \restriction [Q]^2 \equiv 0 \ \land\ Q \notin C_n\bigr).
\]
Fix $(A,c) \in \CW$ and consider two cases.

\medskip
\noindent\textbf{Case 1:} $R_n(A,c)$ fails for every $n$.

We shall prove that $A$ is nowhere $0$-homogeneous for $c$   and we therefore set $H(A,c)=A$.
To see this, let $B\subseteq A$ be $0$-homogeneous. For each $n$, there is $m_n$ such that every finite $0$-homogeneous subset of $A\setminus m_n$ belongs to $C_n$.
In particular, every finite subset of $B\setminus m_n$ belongs to $C_n$.
Since $C_n$ is closed, $B\setminus m_n\in C_n$.
This holds for every $n$, so $B\in\CI$ by \eqref{eq:farah-representation}. Thus $A$ is nowhere $0$-homogeneous.

\medskip
\noindent\textbf{Case 2:} $R_n(A,c)$ holds for some $n$.

Let $n_*$ be the least such $n$, and fix a well-ordering $\lhd$ of $[\omega]^{<\omega}$ of order type $\omega$.
Recursively let $Q_j$ be the $\lhd$-least finite set $Q$ such that
\begin{itemize}
    \item $Q \subseteq A$, and $\min Q > \max Q_{j-1}$ if $j \geq 1$;
    \item $c \restriction [Q]^2 \equiv 0$;
    \item $Q \notin C_{n_*}$.
\end{itemize}
Such a set exists by $R_{n_*}(A,c)$.
Since $\varnothing \in \CI$, we have $\varnothing\in C_{n_*}$; this ensures that every $Q_j$ is nonempty.
So the sets $Q_j$ are nonempty, finite, $0$-homogeneous, and $\max Q_j < \min Q_{j+1}$.

For each $j$, let $v(j)=\min Q_j$, and for $j<k$ set
\[
D(j,k)=\set{x\in Q_j}{c(x,v(k))=0}.
\]
We now recursively construct infinite sets $M_\ell \subseteq \omega$, integers $n_\ell$, and sets $X_\ell \subseteq Q_{n_\ell}$.
Set $M_0=\omega$. 
Given $M_\ell$, let $n_\ell=\min M_\ell$ and $X_\ell$ be the $\lhd$-least $X \subseteq Q_{n_\ell}$ such that
\[
\set{k \in M_\ell}{k > n_\ell \ \land\ D(n_\ell,k) = X} \text{ is infinite.}
\]
Such an $X$ exists because $M_\ell$ is infinite and $Q_{n_\ell}$ has only finitely many subsets.
Set
\[
M_{\ell+1} = \set{k \in M_\ell}{k > n_\ell \ \land\ D(n_\ell,k) = X_\ell}.
\]
Then $(n_\ell)_\ell$ is strictly increasing, and since $n_{\ell'} \in M_{\ell'} \subseteq M_{\ell+1}$ for $\ell < \ell'$,
\begin{equation}\label{eq:farah-stabilized-pattern}
D(n_\ell, n_{\ell'}) = X_\ell
\qquad (\ell < \ell').
\end{equation}
The construction is illustrated in
Figure~\ref{fig:finite-pattern-stabilization}.

\begin{figure}[H]
\centering
\resizebox{0.95\linewidth}{!}{%
\begin{tikzpicture}[
    x=1cm,
    y=1cm,
    block/.style={
        draw,
        minimum width=1.7cm,
        minimum height=.7cm,
        inner sep=0pt
    },
    kept/.style={
        block,
        line width=1.1pt
    },
    discarded/.style={
        block,
        dashed
    },
    point/.style={
        circle,
        draw,
        fill=white,
        inner sep=1.25pt
    },
    xpoint/.style={
        circle,
        draw,
        fill=black,
        inner sep=1.25pt
    },
    marker/.style={
        rectangle,
        draw,
        fill=black,
        minimum size=4pt,
        inner sep=0pt
    }
]

% The underlying order on omega.
\draw[->] (0,0) -- (12.7,0) node[right] {$\omega$};

% The distinguished block Q_{n_\ell}.
\node[kept] at (1.25,0) {};

% The second and third points represent X_\ell.
\node[point]  at (.72,0) {};
\node[xpoint] at (1.05,0) {};
\node[xpoint] at (1.38,0) {};
\node[point]  at (1.72,0) {};

% Bracket indicating X_\ell.
\draw (.94,.22) -- (.94,.43) -- (1.49,.43) -- (1.49,.22);
\coordinate (xell) at (1.215,.43);
\node at (1.215,.68) {$X_\ell$};

\node at (1.25,-.58) {$Q_{n_\ell}$};
\node[font=\scriptsize] at (1.25,-.92)
    {$n_\ell=\min M_\ell$};

% A block discarded at the next stabilization step.
\node[discarded] at (3.55,0) {};
\node[marker] at (3.03,0) {};
\node[point]  at (3.38,0) {};
\node[point]  at (3.70,0) {};
\node[point]  at (4.02,0) {};

\node at (3.55,-.58) {$Q_{r_0}$};
\node[font=\scriptsize] at (3.55,-.92)
    {$r_0\in M_\ell\setminus M_{\ell+1}$};

% A retained block.
\node[kept] at (5.90,0) {};
\node[marker] (vkzero) at (5.38,0) {};
\node[point] at (5.73,0) {};
\node[point] at (6.05,0) {};
\node[point] at (6.37,0) {};

\node[font=\scriptsize, above] at (5.38,0) {$v(k_0)$};
\node at (5.90,-.58) {$Q_{k_0}$};
\node[font=\scriptsize] at (5.90,-.92)
    {$k_0\in M_{\ell+1}$};

% Another discarded block.
\node[discarded] at (8.25,0) {};
\node[marker] at (7.73,0) {};
\node[point] at (8.08,0) {};
\node[point] at (8.40,0) {};
\node[point] at (8.72,0) {};

\node at (8.25,-.58) {$Q_{r_1}$};
\node[font=\scriptsize] at (8.25,-.92)
    {$r_1\in M_\ell\setminus M_{\ell+1}$};

% A second retained block.
\node[kept] at (10.60,0) {};
\node[marker] (vkone) at (10.08,0) {};
\node[point] at (10.43,0) {};
\node[point] at (10.75,0) {};
\node[point] at (11.07,0) {};

\node[font=\scriptsize, above] at (10.08,0) {$v(k_1)$};
\node at (10.60,-.58) {$Q_{k_1}$};
\node[font=\scriptsize] at (10.60,-.92)
    {$k_1\in M_{\ell+1}$};

% Stabilized color-0 connections.
\draw (xell) to[bend left=18]
    node[pos=.55, above] {$0$} (vkzero.north);

\draw (xell) to[bend left=15]
    node[pos=.68, above] {$0$} (vkone.north);

\node at (12.05,0) {$\cdots$};

\end{tikzpicture}%
}
\caption{The construction of $X_\ell$ and $M_{\ell+1}$.}
\label{fig:finite-pattern-stabilization}
\end{figure}
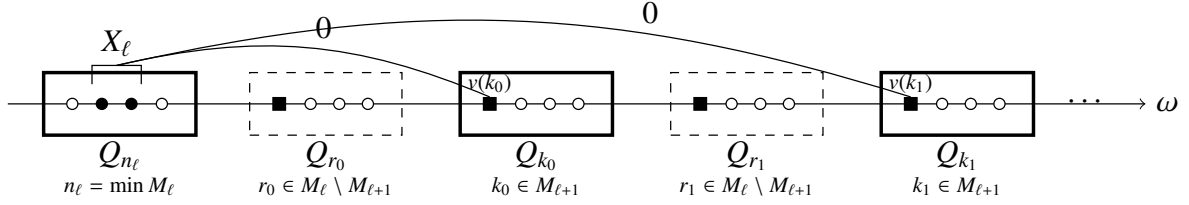
Set 
\[
X=\bigcup_\ell X_\ell \qquad Y=\bigcup_\ell(Q_{n_\ell}\setminus X_\ell).
\]
We show that $X$ is $0$-homogeneous for $c$, that $Y$ is nowhere $0$-homogeneous for $c$, and that $X\cup Y\in\CI^+$.

To see that $X$ is $0$-homogeneous, first note that each
$X_\ell$ is $0$-homogeneous, being a subset of $Q_{n_\ell}$.
If $x\in X_\ell$ and $y\in X_{\ell'}$ with $\ell<\ell'$,
choose $r>\ell'$.
By \eqref{eq:farah-stabilized-pattern},
$c(x,v(n_r))=c(y,v(n_r))=0$.
Since $x<y<v(n_r)$, \eqref{eq:common-right-endpoint}
gives $c(x,y)=0$.
Thus $X$ is $0$-homogeneous.

To see that $Y$ is nowhere $0$-homogeneous, we first claim
that points of $Y$ in different blocks have color $1$.
Indeed, let $x\in Y\cap Q_{n_\ell}$ and
$y\in Y\cap Q_{n_{\ell'}}$, where $\ell<\ell'$,
and suppose that $c(x,y)=0$.
Since $x\notin X_\ell=D(n_\ell,n_{\ell'})$,
we have $c(x,v(n_{\ell'}))=1$.
Thus $y\ne v(n_{\ell'})$, and hence
$x<v(n_{\ell'})<y$.
Both $v(n_{\ell'})$ and $y$ belong to the
$0$-homogeneous set $Q_{n_{\ell'}}$, so
$c(v(n_{\ell'}),y)=0$.
Subadditivity now gives
\[
c(x,v(n_{\ell'}))
\leq c(x,y)+c(v(n_{\ell'}),y)=0.
\]
It follows that $x\in D(n_\ell,n_{\ell'})=X_\ell$,
contrary to $x\notin X_\ell$.
This proves the claim.

Consequently, every $0$-homogeneous subset of $Y$
is contained in a single finite block $Q_{n_\ell}$
and hence belongs to $\CI$.
Thus $Y$ is nowhere $0$-homogeneous.

We have $X \cup Y = \bigcup_\ell Q_{n_\ell}$.
For every $m \in \omega$ there is $\ell$ with $\min Q_{n_\ell} \geq m$, and then
\[
(X \cup Y) \setminus m \supseteq Q_{n_\ell} \notin C_{n_*}.
\]
As $C_{n_*}$ is hereditary, $(X \cup Y) \setminus m \notin C_{n_*}$ for every $m$, so $X \cup Y \notin \CI$ by~\eqref{eq:farah-representation}.

Now define
\[
H(A,c)=
\begin{cases}
X,&\text{if } X\notin\CI,\\
Y,&\text{if } X\in\CI.
\end{cases}
\]
In both cases $H(A,c) \subseteq A$.
If $X \notin \CI$, then $H(A,c)$ is $\CI$-positive and $0$-homogeneous.
If $X \in \CI$, then $Y \notin \CI$ because $X \cup Y \notin \CI$, and $Y$ is nowhere $0$-homogeneous for $c$.

The case distinction and all choices in the construction
are Borel, so $H$ is Borel.
\end{proof}

\section{Applications to closed \texorpdfstring{$\CI$}{I}-MED families}\label{sec:applications}

\begin{theorem}\label{thm:closed-MED-applications}
Every ideal $\CI$ in either of the following classes
admits a closed $\CI$-MED family:
\begin{enumerate}[label=\textup{(\roman*)}]
\item\label{item:MED-Sigma3}
$\boldsymbol{\Sigma}^0_3$ ideals;
\item\label{item:MED-Farah}
weakly Farah ideals, including all analytic $P$-ideals.
\end{enumerate}
\end{theorem}

\begin{proof}
Both classes consist of Borel ideals and therefore have
the Baire property.
Since $\fin\subseteq\CI$, the conclusion follows from
\cite[Theorems~4.8 and~6.2]{HLSZ} once we verify that
$\CI$ is uniformly weakly Ramsey.

For~\ref{item:MED-Sigma3}, Theorem~\ref{thm:Sig03-uniform-p+} shows that $\CI$ is uniformly weakly $P^+$.
Since $\CI$ is Borel, \cite[Lemma~5.2]{HLSZ} applies and implies that $\CI$ is uniformly weakly Ramsey.

For~\ref{item:MED-Farah}, $\CI$ is uniformly weakly Ramsey by Theorem~\ref{thm:weakly-farah-ideals-uwr}.
The assertion for analytic $P$-ideals follows since every analytic $P$-ideal is Farah \cite[Theorem~3.13 and the following observation]{HM2011}, and hence weakly Farah.
\end{proof}

\section*{Acknowledgment}
Jialiang He and Jiaheng Zhuang are supported by the
National Natural Science Foundation of China
(Grant No.~12671002).

\end{document}